\documentclass[12pt,a4paper]{article}%
\usepackage[utf8]{inputenc}
\usepackage{hyperref}
\usepackage{amsmath}
\usepackage{amsfonts}
\usepackage{amssymb}
\usepackage{xcolor}
\usepackage{graphicx}%
\graphicspath{{Slike/}}
\usepackage{tikz}
\usepackage{float}
\providecommand{\U}[1]{\protect\rule{.1in}{.1in}}
\newtheorem{theorem}{Theorem}

\newtheorem{conjecture}[theorem]{Conjecture}
\newtheorem{corollary}[theorem]{Corollary}

\newtheorem{definition}[theorem]{Definition}

\newtheorem{lemma}[theorem]{Lemma}

\newtheorem{observation}[theorem]{Observation}

\newenvironment{proof}[1][Proof]{\noindent\textbf{#1.} }{\ \hfill \rule{0.5em}{0.5em}\bigskip}
\graphicspath{{Slike/}}

\begin{document}

\title{Counterexamples to two conjectures on $(1,2)$-domination in cubic
graphs}
\author{Martin Knor$^{1}$, Jelena Sedlar$^{2,4}$, Riste \v{S}krekovski$^{3,4,5}$\\[0.3cm]
{\small $^{1}$ \textit{Slovak University of Technology in Bratislava, Faculty of Civil
Engineering,}}\\[0.1cm]
{\small \textit{Department of Mathematics, Bratislava, Slovakia}}\\[0.1cm]
{\small $^{2}$ \textit{University of Split, Faculty of Civil Engineering, Architecture and
Geodesy, Split, Croatia}}\\[0.1cm]
{\small $^{3}$ \textit{University of Ljubljana, Faculty of Mathematics and Physics, Ljubljana,
Slovenia}}\\[0.1cm]
{\small $^{4}$ \textit{University of Novo Mesto, Faculty of Information Studies, Novo Mesto,
Slovenia}}\\[0.1cm]
{\small $^{5}$ \textit{Rudolfovo -- Science and Technology Centre Novo Mesto, Novo Mesto,
Slovenia}}}
\date{}
\maketitle

\begin{abstract}
Let $G$ be a cubic graph of order $n$. The induced cycles vertex number $c_{\mathrm{ind}}(G)$ is the largest size of a vertex set that induces a $2$-regular subgraph of $G$.
By $\gamma_{1,2}(G)$ we denote the $(1,2)$-domination number of $G$.
Erve\v{s} and Tepeh introduced the \emph{trilobite} graphs $T_{n}$, which satisfy $\gamma_{1,2}(T_{n})>c_{\mathrm{ind}}(T_{n})$.
They stated two conjectures, the first of which says that every cubic graph $G$ with $c_{\mathrm{ind}}(G)\ge n/2+2$ satisfies $\gamma_{1,2}(G)\leq c_{\mathrm{ind}}(G)$.
The second says that a connected cubic graph $G$ satisfies $\gamma_{1,2}(G)>c_{\mathrm{ind}}(G)$ if and only if $G$ is a trilobite.
We show that both conjectures are false.
A computer search finds counterexamples that are not trilobites already for $n=18,$ $20$ and $22$.
We also construct an infinite family $H(k)$ of order $n=16+4k$.
For every $k\ge 1$ we prove that $c_{\mathrm{ind}}(H(k))=n/2+2$ and $\gamma_{1,2}(H(k))=n/2+3$.
Hence both conjectures fail for infinitely many orders $n$.
\end{abstract}

\textit{Keywords:} Cubic graphs; $(1,2)$-domination; Induced cycles vertex number;
Counterexample; Trilobite graphs.

\textit{AMS Subject Classification numbers:} 05C38, 05C69

\section{Introduction}

Throughout this paper all graphs are finite and simple.
A graph is \emph{cubic} if each of its vertices has degree $3$.
Let $G$ be a graph of order $n$. A set $S\subseteq V(G)$ is \emph{good} if the induced subgraph $\langle S\rangle$ is $2$-regular.
Equivalently, $S$ induces a disjoint union of cycles of $G$.
The \emph{induced cycles vertex number} $c_{\mathrm{ind}}(G)$ is the largest size
of a good set.
This parameter is hard to compute, since finding a largest induced $r$-regular
subgraph is NP-hard for every fixed $r$~\cite{Cardoso2007}.
For cubic graphs the main question is how small $c_{\mathrm{ind}}(G)$ can be.
Henning, Joos, L\"{o}wenstein and Sasse~\cite{Henning2016}
proved that $c_{\mathrm{ind}}(G)\geq n/(2(r-1))+1/((r-1)(r-2))$ for every $r$-regular graph $G$, which gives $c_{\mathrm{ind}}(G)\geq(n+2)/4$ for $r=3$.
However, they conjectured in~\cite{Henning2016} that the lower bound for cubic graphs is twice as large.

\begin{conjecture}
\label{Con_HJLS}
If $G$ is a cubic graph of order $n$, then $c_{\mathrm{ind}}(G)\geq n/2.$
\end{conjecture}

Conjecture~\ref{Con_HJLS} is known to hold for several restricted classes of cubic graphs.
For claw-free graphs the same authors proved the asymptotically sharp bound
$c_{\mathrm{ind}}(G)>13n/20$.
A survey of the claw-free class can be found in~\cite{Flandrin1997}. Another such class is that of $k$-chordal graphs.
A graph is \emph{$k$-chordal} if it contains no induced cycle longer than $k$. 
Henning, Joos, L\"{o}wenstein and Rautenbach~\cite{Henning2016b} proved that $c_{\mathrm{ind}}(G)\geq5n/8+3/4$ for every connected cubic $4$-chordal graph other than $K_{4}$, $K_{3,3}$ and $K_{2}\square K_{3}$.
But in general, the gap between $(n+2)/4$ and $n/2$ remains open.

One way to close this gap is to bound $c_{\mathrm{ind}}(G)$ from below by another graph parameter.
Such a parameter is of use only if its lower bound is already known to be of order $n/2$.
The $(1,2)$-domination number is a natural candidate.
A set $S\subseteq V(G)$ is a $(1,2)$\emph{-dominating set} if every vertex outside $S$ has at least one neighbour in $S$ and every vertex of $S$ has at least two neighbours in $S$.
(We remark that there are also other types of dominations denoted as $(1,2)$-domination.
However, in this paper we use the definition above, since it has connections to induced cycle vertex number, see below.)
The $(1,2)$\emph{-domination number} $\gamma_{1,2}(G)$ is the smallest size of a $(1,2)$-dominating set.
Domination parameters of cubic graphs form a well-developed subject of their own, see for example~\cite{HenningLowenstein2012}.
For $(1,2)$-domination the basic bounds are due to Fakhran, Gorzin, Henning, Jafari and Touserkani~\cite{Fakhran2021}, who proved that
$$n/2\leq\gamma_{1,2}(G)\leq3n/4$$ for every connected cubic graph $G$.
The lower bound here is exactly the quantity that Conjecture~\ref{Con_HJLS} asks for. Consequently every cubic graph $G$ with
\begin{equation}
\gamma_{1,2}(G)\leq c_{\mathrm{ind}}(G)\label{For_chain}%
\end{equation}
satisfies Conjecture~\ref{Con_HJLS}, since then $n/2\leq\gamma_{1,2}(G)\le c_{\mathrm{ind}}(G)$.
Inequality~(\ref{For_chain}) holds whenever some good set dominates $G$. Indeed, a good set $S$ with $N_{G}[S]=V(G)$ is itself a $(1,2)$-dominating set.
In general, the comparison of the two parameters is much less obvious.

This comparison was taken up by Erve\v{s} and Tepeh.
They showed that the route just described cannot work for all cubic graphs.
For every $n\geq10$ they construct a cubic graph $T_{n}$ called the $n$\emph{th trilobite}.
It consists of three parallel paths (called strands) joined at both ends by a $K_{2,3}$ or a triangle.
Along their length the paths are joined by claw centres, single vertices each forming a center of $K_{1,3}$, with the three strand vertices.
For this family
$$
c_{\mathrm{ind}}(T_{n})=n/2+1\qquad\text{and}\qquad\gamma_{1,2}(T_{n})=n/2+2,
$$
so inequality~(\ref{For_chain}) fails for every trilobite. The trilobites are nevertheless harmless for Conjecture~\ref{Con_HJLS}, since $c_{\mathrm{ind}}(T_{n})$ still exceeds $n/2$.
This suggested that the graphs violating~(\ref{For_chain}) are rare and can be described completely.
The trilobites and further computational evidence led them to propose
in~\cite{ErvesTepeh2026} the following two conjectures.
Taken together, these would imply Conjecture~\ref{Con_HJLS}.

\begin{conjecture}
\label{Con_ET2}
If a cubic graph $G$ satisfies $c_{\mathrm{ind}}(G)\geq n/2+2,$ then
$\gamma_{1,2}(G)\leq c_{\mathrm{ind}}(G).$
\end{conjecture}

\begin{conjecture}
\label{Con_ET3}
For a connected cubic graph $G$ we have $\gamma_{1,2}(G)>c_{\mathrm{ind}}(G)$ if and only if $G$ is a trilobite.
\end{conjecture}

In this paper we show that both of these conjectures are false.
We first carry out an exhaustive computer search over all connected cubic graphs of order $n\leq22$.
Non-trilobite graphs violating~(\ref{For_chain}) appear already at $n=18$.
Besides the trilobite itself there are three such graphs of order $18$, four of order $20$ and fifteen of order $22$.
Several of them violate Conjecture~\ref{Con_ET2} as well.
Some of the examples of order $22$ even satisfy $c_{\mathrm{ind}}(G)=n/2+3$.
Raising the threshold in Conjecture~\ref{Con_ET2} to $n/2+3$ would
therefore not repair it.
We then construct an explicit family $H(k)$ of order $n=16+4k$ with $k\geq1$.
Its members resemble the trilobites, but one of the two end caps is replaced by a
branching gadget.
For this family we prove that
$$
c_{\mathrm{ind}}(H(k))=n/2+2\qquad\text{and}\qquad
\gamma_{1,2}(H(k))=n/2+3\qquad\text{for every }
k\ge 1.
$$
Hence both conjectures fail for infinitely many orders $n$.
All our examples satisfy Conjecture~\ref{Con_HJLS}.
The original conjecture of Henning et al.\ is therefore untouched.
What fails is only the proposed route to it through $\gamma_{1,2}$.

\section{$(1,2)$-domination number of $H(k)$}\label{Sec_gamma12}

Throughout the paper we use the notation of~\cite{ErvesTepeh2026}.
For $S\subseteq V(G)$ we write $N_{G}[S]=S\cup\bigcup_{v\in S}N_{G}(v)$.
The following observation is essentially Observation~1 of~\cite{ErvesTepeh2026}.

\begin{observation}
\label{Obs_dom}
If $S$ is a good set of a cubic graph $G$ with $N_{G}[S]=V(G),$ then $S$ is a
$(1,2)$-dominating set of $G,$ and consequently $\gamma_{1,2}(G)\le c_{\mathrm{ind}}(G)$.
\end{observation}

The precise definition of the trilobites $T_{n}$ is not needed here.
We use only the following description from~\cite{ErvesTepeh2026}.
A trilobite consists of three vertex-disjoint paths of length roughly $n/4$, called \emph{strands}.
The strands are joined at both ends by a $K_{2,3}$ or a triangle.
At each interior position they are joined by a \emph{claw centre}, which is a new
vertex adjacent to the three strand vertices at that position.
Our family $H(k)$ reuses this claw layer as a building block.
The difference is that one of the two end caps is replaced by a small gadget in which the strands branch.

\begin{figure}[h]
\begin{center}%
\begin{tabular}
[c]{c}%
\includegraphics[scale=0.665]{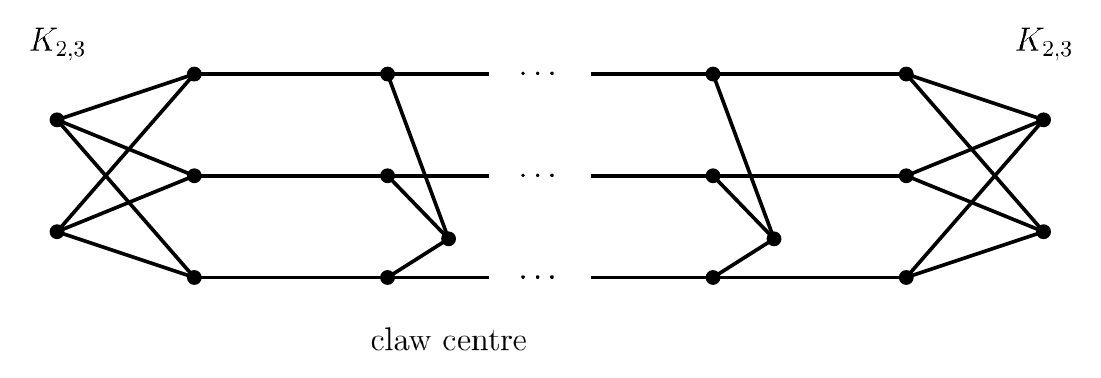}\\
(a) a trilobite $T_{n}$\\[0.45cm]%
\includegraphics[scale=0.665]{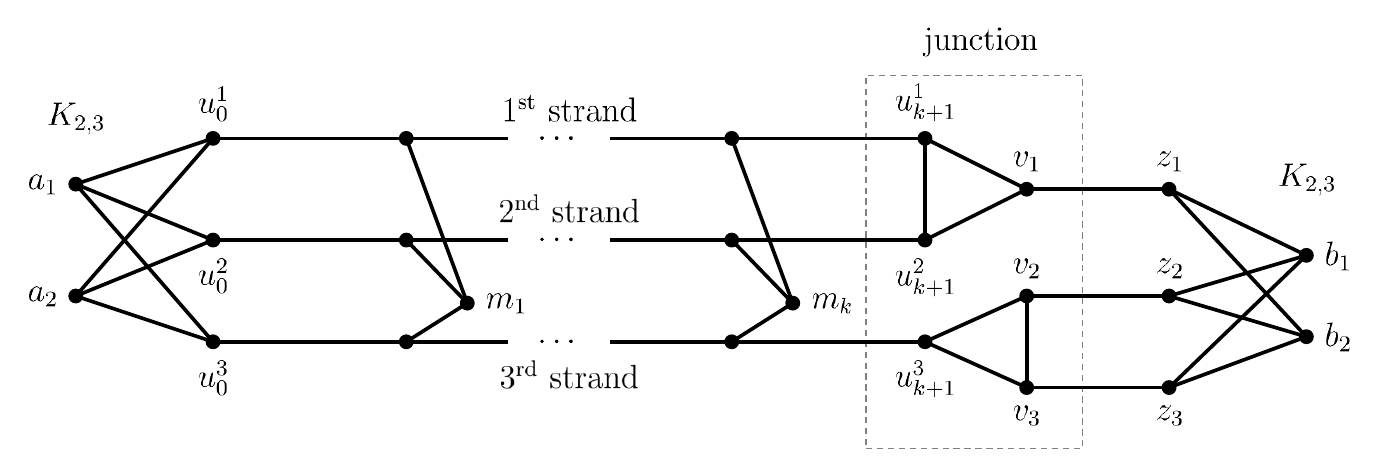}\\
(b) the graph $H(k)$%
\end{tabular}
\end{center}
\caption{Both families are built from the same claw layers.
In the trilobite shown both ends are $K_{2,3}$ caps, while in other trilobites either end may instead be a triangle.
In $H(k)$ a junction is inserted before the second cap.
There the first and second strands merge through $v_1$, while the third strand splits into two new branches through $v_2$ and $v_3$.}%
\label{Fig_compare}%
\end{figure}

\begin{definition}
\label{Def_Hk}
Let $k\ge 1$.
The graph $H(k)$ has vertex set
\begin{align*}
V(H(k))=&\underbrace{\{a_1,a_2,u^1_{0},u^2_{0},u^3_{0}\}}_{\text{cap }A}\ \cup\ \underbrace{\bigcup
_{i=1}^{k}\{u^1_{i},u^2_{i},u^3_{i},m_{i}\}}_{\text{claw layers}}\ \cup\ \underbrace{\{u^1_{k+1},u^2_{k+1},u^3_{k+1},v_1,v_2,v_3\}}_{\text{junction}}\\
& \cup\ \underbrace{\{z_{1},z_{2},z_{3},b_1,b_2\}}_{\text{cap }B},
\end{align*}
and its edges are as follows.

\begin{itemize}
\item
\emph{Cap} $A$ consists of six edges $a_1u^1_{0}$, $a_1u^2_{0}$, $a_1u^3_{0}$, $a_2u^1_{0}$, $a_2u^2_{0}$ and $a_2u^3_{0}$.
They form a $K_{2,3}$ between $\{a_1,a_2\}$ and $\{u^1_{0},u^2_{0},u^3_{0}\}$.

\item
The \emph{claw layers} contribute, for $1\leq i\le k$, the edges $u^1_{i-1}u^1_{i}$, $u^2_{i-1}u^2_{i}$, $u^3_{i-1}u^3_{i}$ and $m_{i}u^1_{i}$, $m_{i}u^2_{i}$, $m_{i}u^3_{i}$.
The set $L_{i}=\{u^1_{i},u^2_{i},u^3_{i},m_{i}\}$ is the $i$th claw layer.
The first three edges continue the three \emph{strands}, while the last three attach the claw centre $m_{i}$ to them.

\item
The \emph{junction} is attached to the last claw layer by the edges $u^1_{k}u^1_{k+1}$, $u^2_{k}u^2_{k+1}$ and $u^3_{k}u^3_{k+1}$, and it contains two triangles.
The triangle $u^1_{k+1}u^2_{k+1}v_1$ merges the first two strands into a single new branch through $v_1$.
The triangle $u^3_{k+1}v_{2}v_{3}$ splits the third strand into two new branches through $v_{2}$ and $v_{3}$.

\item
\emph{Cap} $B$ consists of edges $b_1z_{1}$, $b_1z_{2}$, $b_1z_{3}$, $b_2z_{1}$, $b_2z_{2}$ and $b_2z_{3}$, which form a second $K_{2,3}$ between $\{b_1,b_2\}$ and $\{z_{1},z_{2},z_{3}\}$.
It is attached to the junction by the edges $v_1z_1$, $v_2z_2$ and $v_3z_3$.
\end{itemize}
\end{definition}

Figure~\ref{Fig_compare} shows the difference between a trilobite and a member of $H(k)$.
It is straightforward to check that $H(k)$ is a connected cubic graph of order $n=16+4k$.
Its smallest member $H(1)$ is one of the four non-trilobite counterexamples of order $20$ found by our computer search.
For $k=1,2,3$ we determined both parameters by exhaustive enumeration of all
$2^{n}$ subsets of $V(H(k))$.
In each case $c_{\mathrm{ind}}(H(k))=n/2+2$ and $\gamma_{1,2}(H(k))=n/2+3$.
The remaining counterexamples of order at most $22$ do not belong to the
family.
They are structurally similar to $H(1)$, but their junction gadget is different.

The exact value of $\gamma_{1,2}(H(k))$ is determined in Theorem~{\ref{Tm_gamma12}} at the end of this section.
The companion result for $c_{\mathrm{ind}}(H(k))$ is proved in
Section~{\ref{Sec_cind}}.
Now we introduce some notation, which is used in both sections.

\medskip

\noindent\textbf{Notation.}
In accordance with the names of vertices, we set $L_0=\{u^1_0,u^2_0,u^3_0\}$ and $L_{k+1}=\{u^1_{k+1},u^2_{k+1},u^3_{k+1}\}$.
So $L_0$ is a part of cap~A, and $L_{k+1}$ is a part of the junction.
Vertices of the junction and cap B together form the \emph{tail} $T$.
Let $S$ be a $\gamma_{1,2}$-set of $H(k)$.
We write $\lambda_{i}=|S\cap L_{i}|$ for $0\le i\leq k$ and $\tau=|S\cap T|$.
The two triangles of the junction are denoted by $\Delta_{1}=\{u^1_{k+1},u^2_{k+1},v_1\}$ and $\Delta_{2}=\{u^3_{k+1},v_{2},v_{3}\}$, and we set $\delta_{j}=|S\cap\Delta_{j}|$ for $j=1,2$.
Inside cap $B$ we set $\zeta=|S\cap\{z_{1},z_{2},z_{3}\}|$ and $\beta=|S\cap\{b_1,b_2\}|$.
We also set $\alpha=|S\cap\{a_1,a_2\}|$ in cap~A.
Observe that $|S|=\alpha+\lambda_0+\sum_{i=1}^k \lambda_i+\delta_1+\delta_2+\zeta+\beta$.
The predecessors of $z_{1}$, $z_{2}$ and $z_{3}$ are $v_{1}$, $v_{2}$ and $v_3$, respectively.

We now collect a series of local lemmas.
They describe how a $(1,2)$-dominating set $S$ may meet each cap, each claw layer, and the tail.

\begin{lemma}
\label{Lemma_capA}
In cap A we have

\begin{enumerate}
\item[(a)]
$\alpha\ge 1$ and $\lambda_0\ge 2$, so that $\alpha+\lambda_0\ge 3$.

\item[(b)]
If exactly one of $a_1,a_2$ lies in $S$, then $u^j_1\in S$ whenever $u^j_0\in S$, $1\le j\le 3$.

\item[(c)]
If $\alpha+\lambda_0=3$, then $\lambda_1\ge 2$.

\item[(d)]
If $u^1_{0},u^2_{0},u^3_{0}\in S,$ then $\alpha+\lambda_0\ge 4$.
\end{enumerate}

\noindent
The same four statements hold for the cap $B=\{b_1,b_2,z_{1},z_{2},z_{3}\}$, with $b_1,b_2$ in place of $a_1,a_2$ and with the predecessors $v_1,v_2,v_3$ in place of successors $u^1_{1},u^2_{1},u^3_{1}$.
In particular $\zeta\ge 2$ and $\beta\ge 1$.
Moreover, if $\beta=1$, then every $z_j\in S$ has its predecessor $v_j$ in $S$, $1\le j\le 3$.
\end{lemma}

\begin{proof}
(a):
Suppose that neither $a_1$ nor $a_2$ lies in $S$.
Since $S$ is a $(1,2)$-dominating set, we have $|S\cap N(a_1)|\ge 1$.
Without loss of generality assume that $u^1_0\in S$.
But since $N(u^1_0)=\{a_1,a_2,u^1_1\}$, the vertex $u^1_0$ cannot have two neighbours in $S$, a contradiction.
Hence $\alpha\ge 1$.
But a vertex of $\{a_1,a_2\}\cap S$ forces $\lambda_0\ge 2$.

(b):
Let $a_1$ be the unique vertex of $\{a_1,a_2\}$ lying in $S$, and let $u^j_0\in S$, $1\le j\le 3$.
Then $u^j_0$ needs a second neighbour in $S$ besides $a_1$, so $u^j_1\in S$.

(c):
If $\alpha+\lambda_0=3$, then $\alpha=1$ and $\lambda_0=2$.
Consequently, $\lambda_1\ge 2$ by (b).

(d):
This statement is a consequence of (a).

The statements for cap $B$ follow by the same arguments.
Indeed, cap B is isomorphic to cap A, with $v_1,v_2,v_3$ playing the role of the strand successors $u^1_{1},u^2_{1},u^3_{1}$.
\end{proof}

\begin{lemma}
\label{Lemma_claw}
For $1\le i\le k$ the following hold.

\begin{itemize}
\item[(a)]
$\lambda_i\ge 1$.

\item[(b)]
Let $\lambda_i=1$ and $S\cap L_{i}=\{u^j_i\}$.
Then $m_{i}\notin S$ and $u^j_{i-1},u^j_{i+1}\in S$.
Moreover, $u^t_{i-1}\in S$ or $u^t_{i+1}\in S$ for every $t\in\{1,2,3\}\setminus\{j\}$.

\item[(c)]
If $\lambda_i=1$, then $\lambda_{i-1}\ne 2$ and $\lambda_{i+1}\ne 2$.
Here the indices are taken in $[1,k]$.

\item[(d)]
No three consecutive layers have $\lambda=1$.

\item[(e)]
Let $\lambda_{i}=\lambda_{i+1}=1$ with $i,i+1\in[1,k]$.
Then the two singletons lie on the same strand, say $S\cap L_{i}=\{u^j_i\}$ and $S\cap L_{i+1}=\{u^j_{i+1}\}$.
Moreover $u^1_{i-1},u^2_{i-1},u^3_{i-1}\in S$ and $u^1_{i+2},u^2_{i+2},u^3_{i+2}\in S$.
Consequently $\lambda_{i+2}=4$ if $i+2\le k$, and $\lambda_{i-1}=4$ if $i-1\ge 1$.
\end{itemize}
\end{lemma}

\begin{proof}
(a):
Recall that $N(m_{i})=\{u^1_{i},u^2_{i},u^3_{i}\}$.
If $m_{i}\not\in S,$ then $m_{i}$ needs a neighbour in $S$, which proves (a).
Moreover, if $m_i\in S$ then $\lambda_i\ge 3$, which will be used repeatedly below.

(b):
Let $\lambda_{i}=1$ and $S\cap L_{i}=\{ u^j_i\}$ for some $j$, $1\le j\le 3$.
Then $m_{i}\notin S$, which means that $u^j_{i-1},u^j_{i+1}\in S$.
Now let $t\ne j$, $1\le t\le 3$.
The vertex $u^t_i$ is not in $S$, so it needs a neighbour in $S$ among $u^t_{i-1}$, $u^t_{i+1}$ and $m_{i}$.
Since $m_{i}\notin S$, we have $u^t_{i-1}\in S$ or $u^t_{i+1}\in S$.

(c):
Let $\lambda_{i}=1$ and $\lambda_{i+1}=2$.
Then $m_{i+1}\notin S$ since otherwise $\lambda_{i+1}\ge 3$.
So $u^x_{i+1},u^y_{i+1}\in S$ for some $x,y$, where $1\le x<y\le 3$.
But then also $u^x_{i},u^x_{i+2},u^y_i,u^y_{i+2}\in S$, which contradicts $\lambda_i=1$.
The proof of $\lambda_{i-1}\ne 2$ is symmetric.

(d):
Suppose that $\lambda_{i}=\lambda_{i+1}=\lambda_{i+2}=1$, and let $u^j_{i+1}$ be the vertex of $S\cap L_{i+1}$.
By (b) we have $u^j_i,u^j_{i+2}\in S$, which means that $N(u^t_{i+1})\cap S=\emptyset$ for $t\ne j$, a contradiction.

(e):
First part is implied by (b), so we may assume that $u^j_i,u^j_{i+1}\in S$.
Now we apply (b) to layer $L_{i+1}$.
By the first part of (b) we have $u^j_{i+2}\in S$ and by the second part of (b) we have $u^t_{i+2}\in S$ for $t\ne j$.
Hence $u^1_{i+2},u^2_{i+2},u^3_{i+2}\in S$.
Symmetrically $u^1_{i-1},u^2_{i-1},u^3_{i-1}\in S$.
Now since $u^t_{i+1}\notin S$ for $t\ne j$ and $|N(u^t_{i+2})\cap S|\ge 2$, we  get $m_{i+2}\in S$ if $i+2\le k$, and consequently $\lambda_{i+2}=4$.
The case $i-1\ge 1$ is symmetric.
\end{proof}

\medskip

A claw layer may contain a single vertex of $S$, but by Lemma~\ref{Lemma_claw} it then forces its neighbours to have many vertices in $S$.
The following discharging argument exploits this, and its aim is to give
to every claw layer a charge $2$ or bigger.

\medskip

\noindent\textbf{Discharging.}
Call a layer $i\in[1,k]$ \emph{light} if $\lambda_{i}=1$, \emph{normal}
if $\lambda_{i}=2$, and \emph{heavy} if $\lambda_{i}\ge 3$.
Every layer is of one of these three types, since $\lambda_{i}\ge 1$ by Lemma~{\ref{Lemma_claw}}(a).
A maximal (i.e. non-extendable) sequence of layers with the same value of $\lambda$ is called a {\em run}.
By Lemma~{\ref{Lemma_claw}}(c) and~(d), a maximal run of light layers has length $1$ or $2$.
Such a run is never adjacent to a normal layer.
Each layer $i$ starts with the charge $\lambda_{i}$, and charge is then moved by the two rules below.
In these rules by ``layer $0$'' we mean the whole cap $A$ and by ``layer $k+1$'' we mean the tail $T$.

\begin{itemize}
\item[(R1)]
A light run of length $1$ at layer $i$ receives $1/2$ from each of the layers $i-1$
and $i+1$.

\item[(R2)]
A light run of length $2$ at layers $i$ and $i+1$ receives $1$ from layer $i-1$ to layer $i$ and $1$ from layer $i+2$ to layer $i+1$.
\end{itemize}

\noindent
Let $\widehat{\lambda}_{i}$ be the final charge of layer $i$.
Let $\varepsilon_{A}$ and $\varepsilon_{T}$ be the total charge sent out by the cap $A$ and by the tail $T$, respectively.
Each of $\varepsilon_{A}$ and $\varepsilon_{T}$ is $0$, $1/2$ or $1$.
Charge is only moved and never created, so
\begin{equation}
\left\vert S\right\vert =(\alpha+\lambda_0-\varepsilon_{A})+\sum_{i=1}^{k}\widehat{\lambda}_{i}+(\tau-\varepsilon_{T}).
\label{For_star}
\end{equation}

\begin{lemma}
\label{Lemma_charge}
Let $i\in[1,k]$.
Then $\widehat{\lambda}_{i}\ge 2$.
Moreover, if layer $i$ donates charge $1$ to a light run of length $2$, then $\widehat{\lambda}_{i}=3$.
\end{lemma}

\begin{proof}
Consider first a light layer.
Under (R1) it ends with $1+1/2+1/2=2$, and under (R2) it ends with $1+1=2$.
A normal layer is never adjacent to a light run by Lemma~{\ref{Lemma_claw}(c)}, so it donates nothing and keeps its charge $2$.

Let layer $i$ be heavy.
Suppose that it donates charge $1$ to a light run of length $2$.
Say the run is $(i+1,i+2)$ and lies on the $j$-th strand $1\le j\le 3$.
Then $\lambda_{i}=4$ by Lemma~{\ref{Lemma_claw}}(e), so all four vertices of $L_{i}$ are in $S$.
Let $t\ne j$, $1\le t\le 3$.
Since $u^t_i\in S$ and $u^t_{i+1}\notin S$, we have $u^t_{i-1}\in S$.
Since we have two choices for $t$, $\lambda_{i-1}\ge 2$, and layer $i-1$ is not light.
So layer $i$ donates nothing to its left, and its final charge is $\widehat{\lambda}_{i}=4-1=3$.

In every other case a heavy layer donates at most $1/2$ to each side, hence at most $1$ in total.
Its final charge is therefore at least $3-1=2$.
\end{proof}

\begin{lemma}
\label{Lemma_capD}
We have $\alpha+\lambda_0-\varepsilon_{A}\ge 3$.
Moreover, if $\varepsilon_{A}=1$, then $\alpha+\lambda_{0}-\varepsilon_{A}=4$.
\end{lemma}

\begin{proof}
By Lemma~{\ref{Lemma_capA}}(a) we have $\alpha+\lambda_{0}\ge 3$.
Suppose first that $\alpha+\lambda_{0}=3$.
Then $\lambda_{1}\ge 2$ by Lemma~{\ref{Lemma_capA}}(c), so layer $1$ is not light. The cap then sends out nothing, that is, $\varepsilon_{A}=0$.
Suppose next that $\alpha+\lambda_{0}\ge 4$.
Since $\varepsilon_{A}\le 1$, we again get $\alpha+\lambda_{0}-\varepsilon_{A}\ge 3$.
This proves the first claim.

For the second claim, let $\varepsilon_{A}=1$.
Then $\alpha+\lambda_{0}\ge 4$ as shown above.
The cap sends out charge $1$ only under (R2), so the layers $1$ and $2$ form a light run of length $2$.
Applying Lemma~{\ref{Lemma_claw}}(e) to this run gives $u^1_{0},u^2_{0},u^3_{0}\in S$, so $|L_0\cap S|=3$.
Suppose that exactly one from $a_1,a_2$ is in $S$.
Then Lemma~{\ref{Lemma_capA}}(b) gives $u^1_1,u^2_1,u^3_1\in S$ and so $\lambda_{1}\ge 3$, a contradiction.
Hence $a_1,a_2\in S$, which gives $\alpha+\lambda_{0}=5$ and $\alpha+\lambda_{0}-\varepsilon_{A}=4$.
\end{proof}

\begin{lemma}
\label{Lemma_tail}
We have $\delta_{1},\delta_{2}\in\{0,2,3\}$, and the following hold.

\begin{itemize}
\item[(E1)]
$\tau\ge 5$.

\item[(E2)]
$\tau=5$ if and only if $\delta_{1}=\delta_{2}=0$.
In that case $\lambda_{k}=4$ and $\lambda_{k-1}\ge 3$.

\item[(E3)]
If $\tau=6$, then $\lambda_{k}=4$.

\item[(E4)]
If $\tau=7$, then either $\lambda_{k}\ge 3$, or $\lambda_{k}=2$ and $\lambda_{k-1}=4$.

\item[(E5)]
If $\lambda_{k}=1$, then $\tau\ge 8$.
\end{itemize}
\end{lemma}

\begin{proof}
If a vertex of a triangle is in $S$, then at least one other vertex of this triangle is in $S$ since $S$ is a $(1,2)$-dominating set.
Consequently, $\delta_{1},\delta_{2}\in\{0,2,3\}$.

We next describe forcing arguments that we use below.
Let $1\le i\le 2$.
If $\delta_i=0$ then all vertices of $\Delta_i$ must have a neighbour in $S$.
Consequently, all vertices of $N(\Delta_i)\setminus V(\Delta_i)$ must be in $S$.
On the other hand if $\delta_i=2$, then the vertices of $\Delta_i$ which are in $S$ must have a neighbour outside $\Delta_i$ in $S$.

First we prove (E1) and (E2).
Recall that $\tau=\delta_{1}+\delta_{2}+\zeta+\beta$ and $\zeta+\beta\ge 3$ by
Lemma~{\ref{Lemma_capA}}(a).
Suppose first that $\delta_{1}=\delta_{2}=0$.
Then $z_{1},z_{2},z_{3}\in S$ by the forcing arguments above.
And since $|N(z_j)\cap S|\ge 2$, we have $b_1,b_2\in S$ as well.
Hence $\zeta+\beta=5$ and $\tau=5$.

Suppose next that exactly one of $\delta_{1},\delta_{2}$ is $0$.
Then $\tau\ge 2+0+3=5$, and equality would need $\zeta+\beta=3$.
In that case $\beta=1$, so every $z_j\in S$ has its predecessor $v_j$ in $S$.
This is possible only if $\delta_1=0$, $\delta_2=2$, and $z_2,z_3\in S$.
But then $v_1$ has no neighbour in $S$, a contradiction.

Finally, if neither $\delta_{1}$ nor $\delta_{2}$ is $0$, then $\tau\ge 2+2+3=7$. This proves (E1), and it also shows that $\tau=5$ holds exactly when $\delta_{1}=\delta_{2}=0$.

Let $\tau=5$.
Then $\delta_1=\delta_2=0$, and so $u^1_{k},u^2_{k},u^3_{k}\in S$ by the forcing arguments.
Each of $u^1_{k},u^2_{k},u^3_{k}$ therefore needs its predecessor and $m_{k}$ in $S$.
This gives $\lambda_{k}=4$ and $\lambda_{k-1}\ge 3$, which completes (E2).

(E3):
Let $\tau=6$.
As shown above, exactly one of $\delta_{1},\delta_{2}$ is $0$, say
$\delta_{x}=0$.
Denote by $y$ the value such that $\{x,y\}=\{1,2\}$.
We distinguish two cases.

{\bf Case 1}: $\delta_y=3$.
Then $\zeta+\beta=3$.
Consequently two $z$'s are in $S$, and since $\beta=1$, both its predecessors are in $S$ as well.
Consequently $y=2$ and $S\cap T$ contains $u^3_{k+1},v_2,v_3,z_2,z_3$ and one of $b_1,b_2$.
But then no neighbour of $v_1$ is in $S$, a contradiction.

{\bf Case 2}: $\delta_y=2$.
Then $\zeta+\beta=4$.
By the forcing arguments at least $5$ vertices of $u^1_{k},u^2_{k},u^3_{k},z_1,z_2,z_3$ are in $S$.
If $z_1,z_2,z_3\in S$ then $|\{b_1,b_2\}\cap S|=1$, and by Lemma~{\ref{Lemma_capA}}(b) all predecessors of $z_1,z_2,z_3$ are in $S$.
But this contradicts $\delta_x=0$.
Hence $u^1_{k},u^2_{k},u^3_{k}\in S$.
And since at least one of $u^1_{k},u^2_{k},u^3_{k}$ has a neighbour in $\Delta_x$, we have also $m_k\in S$, which gives $\lambda_k=4$.

(E4):
Let $\tau=7$.
The possible splittings $(\delta_{1},\delta_{2},\zeta+\beta)$ are $(0,2,5)$, $(2,0,5)$, $(0,3,4)$, $(3,0,4)$ and $(2,2,3)$.
In the first four splittings there is a triangle, say $\Delta_x$, with $\delta_x=0$.
By forcing arguments there is $u^j_k$ which is in $S$ and its neighbour in $\Delta_x$ is not in $S$.
Consequently, $u^j_{k-1},m_k\in S$.
And since $m_k\in S$, we have $\lambda_k\ge 3$.

In the remaining case $\zeta+\beta=3$, so $\zeta=2$ and $\beta=1$.
Exactly one of $z_{1},z_{2},z_{3}$ lies outside $S$, and by Lemma~{\ref{Lemma_capA}}(b) the predecessors of the other two lie in $S$.
The predecessor of the third one does not, since otherwise this third vertex would be in $S$ by the forcing arguments.
Hence, $|\{v_1,v_2,v_3\}\cap S|=2$, and then also $|\{u^1_{k+1},u^2_{k+1},u^3_{k+1}\}\cap S|=2$.
Consequently, $\lambda_k\ge 2$ by the forcing arguments.
If $\lambda_k\ge 3$, we are done.
So suppose that $\lambda_k=2$.
Since $m_k\notin S$, the predecessors of the two vertices of $L_k\cap S$ must be in $S$.
Also the predecessor of the third vertex $u^j_k$ of $L_k$ must be in $S$, since otherwise $u^j_k$ does not have a neighbour in $S$.
Since $u^j_{k-1}\in S$ and $u^j_k\notin S$, we have $m_{k-1}\in S$, which finally gives $\lambda_{k-1}=4$.

The statement (E5) is a direct consequence of (E1) - (E4).
\end{proof}

\medskip

We are ready to determine $\gamma_{1,2}(H(k))$.

\begin{theorem}
\label{Tm_gamma12}
Let $n=4k+16$, where $k\ge 1$.
Then
$$
\gamma_{1,2}(H(k))=n/2+3.
$$
\end{theorem}

\begin{proof}
In $k$ the statement gives $\gamma_{1,2}(H(k))=2k+11$.

\medskip

First we prove the lower bound.
For $k\le 2$ the bound is confirmed by the exhaustive computation reported after
Definition~\ref{Def_Hk}.
So let $k\ge 3$.
Lemma~{\ref{Lemma_charge}} gives $\sum_{i=1}^{k}\widehat{\lambda}_{i}\ge 2k$, and Lemma~\ref{Lemma_capD} gives $\alpha+\lambda_0-\varepsilon_{A}\ge 3$.
Hence~(\ref{For_star}) yields
\begin{equation}
\left\vert S\right\vert \ge 3+2k+(\tau-\varepsilon_{T}).
\label{For_base}%
\end{equation}
In each case below we sharpen one of the three terms.

\medskip

{\bf Case 1}: $\lambda_{k}\ge 2$.
Then layer $k$ is not light, so $\varepsilon_{T}=0$.
By (E1) we have $\tau\ge5$, so the four subcases below are exhaustive.

Let $\tau\ge 8$.
Then~(\ref{For_base}) gives $|S|\ge 3+2k+8=2k+11$.

Let $\tau=7$.
Then (E4) leaves three possibilities.
First suppose that $\lambda_{k}=3$.
Layer $k$ cannot donate charge $1$ to a run of length $2$, since that would force
$\lambda_{k}=4$ by Lemma~{\ref{Lemma_claw}}(e).
Hence $\widehat{\lambda}_{k}\ge 3-1/2=5/2$ and $\sum\widehat{\lambda}_{i}\ge 2k+1/2$.
Next suppose that $\lambda_{k}=4$.
Then $\widehat{\lambda}_{k}\ge 3$ and $\sum\widehat{\lambda}_{i}\ge 2k+1$.
In the third possibility $\lambda_{k}=2$ and $\lambda_{k-1}=4$, so
$\widehat{\lambda}_{k}=2$ and $\widehat{\lambda}_{k-1}\ge 3$ by Lemma~{\ref{Lemma_charge}}.
Again $\sum\widehat{\lambda}_{i}\ge 2k+1$.
In all three cases $|S|\ge 3+(2k+1/2)+7$.
Hence $|S|\ge 2k+11$ because $|S|$ is an integer.

Let $\tau=6$.
Then $\lambda_{k}=4$ by (E3).
If layer $k-1$ is not light, then $\widehat{\lambda}_{k}=4$ and $|S|\ge 3+(2k+2)+6=2k+11$.
If layer $k-1$ is a light run of length $1$, then $\widehat{\lambda}_{k}=7/2$ and $|S|\ge 3+(2k+3/2)+6$, hence again $|S|\ge 2k+11$.
In the remaining subcase the light run is $(k-2,k-1)$, so $\widehat{\lambda}_{k}=3$. For $k\ge 4$ the other donor is layer $k-3$ with $\widehat{\lambda}_{k-3}=3$ by
Lemma~{\ref{Lemma_charge}}, giving $|S|\ge 3+(2k+2)+6$.
For $k=3$ the other donor is the cap $A$, so $\varepsilon_{A}=1$ and $\alpha+\lambda_{0}-\varepsilon_{A}=4$ by Lemma~{\ref{Lemma_capD}}. This gives
$|S|\ge 4+(2k+1)+6$.

Let $\tau=5$.
Then $\lambda_{k}=4$ and $\lambda_{k-1}\ge 3$ by (E2), so layer $k-1$ is not light and $\widehat{\lambda}_{k}=4$.
Suppose that $\lambda_{k-1}=3$.
As shown in the proof of (E2), $u^1_{k-1},u^2_{k-1},u^3_{k-1}\in S$.
So $m_{k-1}\notin S$ and consequently $u^1_{k-2},u^2_{k-2},u^3_{k-2}\in S$.
Hence layer $k-1$ donates nothing.
Consequently $\widehat{\lambda}_{k-1}=3$.
Now suppose that $\lambda_{k-1}=4$.
Then $\widehat{\lambda}_{k-1}\ge 3$ as well.
Either way $\sum \widehat{\lambda}_{i}\ge 2k+3$, so $|S|\ge 3+(2k+3)+5=2k+11$.

\medskip

{\bf Case 2}: $\lambda_{k}=1$.
Then $\tau\ge 8$ by (E5).
If layer $k$ is a light run of length $1$, then $\varepsilon_{T}=1/2$, and~(\ref{For_base}) gives $|S|\ge 3+2k+8-1/2$.
Hence $|S|\ge 2k+11$.
If instead the light run is $(k-1,k)$, then $\varepsilon_{T}=1$, and the other donor is layer $k-2$.
Here $k-2\ge 1$, so $\widehat{\lambda}_{k-2}=3$ by Lemma~{\ref{Lemma_charge}} and $\sum\widehat{\lambda}_{i}\ge 2k+1$.
Therefore $|S|\ge 3+(2k+1)+8-1=2k+11$.
This exhausts all cases, so $\gamma_{1,2}(H(k))\ge 2k+11$.

\medskip

Now we prove the upper bound.
We present a $(1,2)$-dominating set of the required size.
Put $j=k/2$ for even $k$ and $j=(k+1)/2$ for odd $k$.
Let $\frak C$ be a sequence $\frak C=\{ c(i)\}_{i=0}^j$ which is periodic with period $3$, $c_i\in\{1,2,3\}$ for every $i$, $0\le i\le k$, and c(j)=1, c(j-1)=3 and $c(j-2)=2$.
In both cases we take
$$
S\cap T=\{u^1_{k+1},u^2_{k+1},v_1,v_2,v_3\}\cup\{z_2,z_3,b_2\}\qquad(\tau=8).
$$
If $k=2j$ is even, we take
\begin{align*}
S\cap L_{0}  &  =\{a_2,u^{c(0)}_0,u^{c(1)}_0\},\\
S\cap L_{2i-1}  &  =\{u^{c(i-1)}_{2i-1},u^{c(i)}_{2i-1},m_{2i-1}\},\\
S\cap L_{2i}  &  =\{u^{c(i)}_{2i}\}
\end{align*}
for $1\leq i\leq j$.
If $k=2j-1$ is odd, we take instead
\begin{align*}
S\cap L_{0}  &  =\{a_1,a_2,u^{c(0)}_0,u^{c(1)}_0\},\\
S\cap L_{2i-1}  &  =\{u^{c(i)}_{2i-1}\}\qquad\qquad\qquad(1\leq i\leq j),\\
S\cap L_{2i}  &  =\{u^{c(i)}_{2i},u^{c(i+1)}_{2i},m_{2i}\}\qquad(1\leq i\leq j-1).
\end{align*}

It should be checked that $S$ is a $(1,2)$-dominating set.
First, if $m_i\in S$ then there are two strand vertices in $i$-th layer which are in $S$, while if $m_i\notin S$ then there is one strand vertex in $i$-th layer which is in $S$.
Further, in each strand there are three consecutive vertices in $S$, where the first and last are in layer containing $m$-vertex in $S$, followed by three consecutive vertices not in $S$, where the central one is in a layer having $m$-vertex in $S$.
So they satisfy the conditions for $S$.
This pattern is repeated, and it is violated only at the both ends, but the vertices close to the ends as well as those of cap A and the tail can be checked by hand.
Counting the vertices,
$$
\left\vert S\right\vert = 3+4j+8=2k+11\quad(\text{$k$ even}),
\qquad\left\vert S\right\vert = 4+(4j-3)+8=2k+11\quad(k\text{ odd}),
$$
as required.
\end{proof}

\section{Induced cycle number of $H(k)$}\label{Sec_cind}

In this section we determine the second parameter.

\medskip

\noindent\textbf{Notation.}
We keep the notation of Section~\ref{Sec_gamma12}.
However, $S$ denotes a \emph{good} set of $H(k)$.
That is, $|N(v)\cap S|=2$ for every $v\in S$.
As before we put $\alpha=|S\cap\{a_1,a_2\}|$, $\lambda_{i}=|S\cap L_{i}|$ for $0\le  i\le k$, $\tau=|S\cap T|$, $\delta_{j}=|S\cap\Delta_{j}|$ for $j=1,2$, $\zeta=|S\cap\{z_1,z_2,z_3\}|$ and $\beta=|S\cap\{b_1,b_2\}|$.
Then
$$
\left\vert S\right\vert =\alpha+\lambda_0+\sum_{i=1}^{k}\lambda_{i}+\tau\qquad\text{and}\qquad\tau=\delta_{1}+\delta_{2}+\zeta+\beta.
$$

\medskip

\noindent\textbf{Interface states}.
Fix $i$ and $j$, $0\le i\le k$ and $1\le j\le 3$.
The vertex $u^j_i$ has exactly one neighbour to the right, namely $u^j_{i+1}$.
For $i=0$ its two remaining neighbours are $a_1$ and $a_2$.
For $i\ge 1$ they are $u^j_{i-1}$ and $m_{i}$.
We call these two the \emph{determined} neighbours of $u^j_i$.
They are decided by $S\cap (L_{0}\cup L_1\cup\dots\cup L_{i})$.
If $u^j_{i}\in S$, then it has two neighbours in $S$.
So $u^j_{i}$ has one or two determined neighbours in $S$, depending on whether $u^j_{i+1}$ lies in $S$.
We therefore define the \emph{state} of $u^j_i$ by (the sans-serif labels
$\mathsf{N}_{1},\mathsf{N}_{2}$ below are just names for the states, not instances of the neighbourhood $N(\cdot)$)
$$
s_{i}(j)=\left\{
\begin{array}
[c]{ll}%
\mathsf{X}, & \text{if $u^j_i\notin S$,}\\
\mathsf{N}_{1}, & \text{if $u^j_i\in S$ and $u^j_{i}$ has exactly one determined
neighbour in $S$,}\\
\mathsf{N}_{2}, & \text{if $u^j_{i}\in S$ and $u^j_i$ has two determined
neighbours in $S$},
\end{array}
\right.
$$
and we put $s_{i}=(s_{i}(1),s_{i}(2),s_{i}(3))$.
By the above,
\begin{equation}
\begin{aligned}
s_{i}(j) & =\mathsf{N}_{1}\quad\mbox{if and only if}\quad u^j_i\in S\text{ and }u^j_{i+1}\in S,\\
s_{i}(j) & =\mathsf{N}_{2}\quad\mbox{if and only if}\quad u^j_i\in S\text{ and }u^j_{i+1}\notin S.
\end{aligned}
\label{For_states}
\end{equation}
We abbreviate three kinds of triples.
By $P$ we denote the state $(\mathsf{X},\mathsf{X},\mathsf{X})$.
By $Q$ we denote a state with two coordinates $\mathsf{N}_{1}$ and the third one $\mathsf{X}$.
Likewise, by $R$ we denote a state with two coordinates $\mathsf{N}_{2}$ and the third one $\mathsf{X}$.
Finally we put
$$
\psi(P)=2,\qquad\psi(Q)=3,\qquad\psi(R)=4.
$$

\begin{lemma}
\label{Lemma_capAcind}
We have $(\alpha+\lambda_0,s_0)\in\{(0,P),(3,Q),(4,R)\}$.
In particular $\alpha+\lambda_0-\psi(s_0)\leq 0$.
\end{lemma}

\begin{proof}
Note that $N(a_1)=N(a_2)=\{u^1_0,u^2_0,u^3_0\}$.
So if $a_1$ (or $a_2$) is in $S$, then exactly two of $u^1_0,u^2_0,u^3_0$ are in $S$.
We distinguish three cases.

{\bf Case 1}: Neither $a_1$ nor $a_2$ lies in $S$.
Since the only possible neighbour of $u^j_0$ in $S$ is $u^j_1$, we have $u^j_0\notin S$, $1\le j\le 3$.
So $S\cap L_0=\emptyset$ and $s_0=P$.

{\bf Case 2}: Exactly one of $a_1,a_2$ lies in $S$.
Then exactly two of $u^1_0,u^2_0,u^3_0$ lie in $S$, so $\alpha+\lambda_0=3$.
Each of the two vertices of $L_0$ in $S$ has a unique determined neighbour in
$S$, namely $\{a_1,a_2\}\cap S$.
So both of them have state $\mathsf{N}_1$, while the third vertex of $L_0$ has state    $\mathsf{X}$.
Hence $s_0=Q$.

{\bf Case 3}: Both $a_1$ and $a_2$ are in $S$.
The two neighbours in $S\cap L_0$ are the same for $a_1$ and for $a_2$,
so $\alpha+\lambda_0=4$.
Each of them has both $a_1$ and $a_2$ as determined neighbours in $S$, so both have
state $\mathsf{N}_2$.
The third vertex of $L_0$ is again not in $S$.
Hence $s_0=R$.

The three possible values of $\alpha+\lambda_0-\psi(s_0)$ are $0-2,$ $3-3$ and $4-4$.
\end{proof}

\begin{lemma}
\label{Lemma_clawcind}
Let $1\leq i\leq k$.
Then the following hold.

\begin{itemize}
\item[(a)]
If $m_{i}\in S$, then exactly two of $u^1_i,u^2_i,u^3_i$ lie in $S$.
In particular $\lambda_i=3$.

\item[(b)]
If $u^j_i\in S$ and $u^j_{i-1}\notin S$, then $m_i\in S$.

\item[(c)]
Let $s_{i-1}\in\{P,Q,R\}$.
Then $s_i\in\{P,Q,R\}$ as well, and the only possible transitions $s_{i-1}\to s_i$ are
\begin{align*}
P\to P&\quad\mbox{with $\lambda_i=0$},\quad
P\to Q\quad\mbox{with $\lambda_i=3$},\quad
Q\to Q\quad\mbox{with $\lambda_i=2$},\\
Q\to R&\quad\mbox{with $\lambda_i=3$},\quad
R\to P\quad\mbox{with $\lambda_i=0$}.
\end{align*}
In particular $\lambda_i\in\{0,2,3\}$.
Moreover $\lambda_i=3$ holds if and only if $m_i\in S$.
\end{itemize}
\end{lemma}

\begin{proof}
Part (a) follows from $N(m_{i})=\{u^1_i,u^2_i,u^3_i\}$ and $\left\vert N(m_{i})\cap S\right\vert=2$.
And since $N(u^j_i)=\{u^j_{i-1},m_i,u^j_{i+1}\}$, we have (b).

We now turn to (c), and we distinguish three cases.

{\bf Case 1}: $s_{i-1}=P$.
Then $\{u^1_{i-1},u^2_{i-1},u^3_{i-1}\}\cap S=\emptyset$.
If $m_{i}\notin S$, then $S\cap L_i=\emptyset$ by (b), so $\lambda_i=0$ and $s_i=P$. If $m_i\in S$, then exactly two strand vertices of $L_{i}$ lie in $S$ by (a), so $\lambda_i=3$.
Each of these two vertices has $m_i$ as its only determined neighbour in $S$, hence it has state $\mathsf{N}_{1}$.
The third strand vertex of $L_i$ has state $\mathsf{X}$, so $s_{i}=Q$.

{\bf Case 2}: $s_{i-1}=Q$.
Without loss of generality assume that $s_{i-1}(1)=s_{i-1}(2)=\mathsf{N}_{1}$ and $s_{i-1}(3)=\mathsf{X}$.
Then $u^1_{i},u^2_{i}\in S$.
If $m_{i}\not \in S$, then $u^3_{i}\notin S$ by (b), so $\lambda_i=2$.
Here each of $u^1_{i},u^2_{i}$ has its predecessor as the only determined neighbour in $S$, so both have state $\mathsf{N}_{1}$ and $s_{i}=Q$.
If $m_{i}\in S$, then $u^3_{i}\notin S$ by (a), so $\lambda_{i}=3$.
Now each of $u^1_{i},u^2_{i}$ has two determined neighbours in $S$, namely its predecessor and $m_{i}$.
So both have state $\mathsf{N}_{2}$ and $s_{i}=R$.

{\bf Case 3}: $s_{i-1}=R$.
Without loss of generality assume that $s_{i-1}(1)=s_{i-1}(2)=\mathsf{N}_{2}$ and $s_{i-1}(3)=\mathsf{X}$.
Then $u^1_{i},u^2_{i}\notin S$, so $m_{i}\in S$ is impossible by (a).
Hence $m_{i}\notin S$, and then $u^3_{i}\notin S$ by (b).
Thus $\lambda_{i}=0$ and $s_{i}=P$.
\end{proof}

By Lemma~{\ref{Lemma_clawcind}}(c), $s_{0},s_{1},\dots,s_{k}$ is a sequence which elements are members of three states.
The remaining $20$ of the $27$ states never occur.
This is what replaces the discharging argument used for Theorem~{\ref{Tm_gamma12}}.

\begin{lemma}
\label{Lemma_potential}
For every $i$, $1\leq i\leq k$, we have
$$
\lambda_{i}\le 2+\psi(s_{i})-\psi(s_{i-1}),
$$
with equality unless the transition is $P\to P$.
Consequently
\begin{equation}
\label{For_telescope}
\sum_{i=1}^{k}\lambda_{i}\le 2k+\psi(s_{k})-\psi(s_{0}).
\end{equation}
\end{lemma}

\begin{proof}
By Lemma~{\ref{Lemma_capAcind}}, $s_{0}\in\{P,Q,R\}$, and by Lemma~{\ref{Lemma_clawcind}}(c), every state $s_{i}$ then also lies in $\{P,Q,R\}$. 
So each step $s_{i-1}\to s_{i}$ is one of the five transitions listed there, and we check the inequality directly on each of them:
$$
\begin{array}
[c]{c|ccccc}%
\text{transition} & P\rightarrow P & P\rightarrow Q & Q\rightarrow Q & Q\rightarrow R &
R\rightarrow P\\\hline
\lambda_{i} & 0 & 3 & 2 & 3 & 0\\
2+\psi(s_{i})-\psi(s_{i-1}) & 2 & 3 & 2 & 3 & 0
\end{array}
$$
In every case $\lambda_{i}$ does not exceed the bottom row, with equality except for $P\to P$.
Summing over $i=1,2,\dots,k$ gives~(\ref{For_telescope}).
\end{proof}

\begin{lemma}
\label{Lemma_tailcind}
We have $\tau\le 8$ if $s_{k}=P$, $\tau\le 7$ if $s_{k}=Q$, and $\tau\le 5$ if $s_{k}=R$.
All three bounds are attained.
In particular $\psi(s_{k})+\tau\le 10$.
\end{lemma}

\begin{proof}
We first present two observations.

\medskip

\noindent(C1)
Consider the triangles $\Delta_{1}$ and $\Delta_{2}$.
Each of their vertices has exactly one neighbour outside the triangles: $u^1_{k},u^2_{k},z_{1}$ for $u^1_{k+1},u^2_{k+1},v_1$, and
$u^3_{k},z_{2},z_{3}$ for $u^3_{k+1},v_{2},v_{3}$.
If a triangle vertex $x$ lies in $S$ while the other two do not, then $x$ has at most one neighbour in $S$, which is impossible.
Hence $\delta_{j}\in\{0,2,3\}$.
If $\delta_{j}=3$, then every triangle vertex already has two neighbours in $S$
inside the triangle, so none of the three outside neighbours lies in $S$.
If $\delta_{j}=2$, then each of the two triangle vertices in $S$ has exactly one neighbour in $S$ inside the triangle, so its outside neighbour is in $S$ too.

\medskip

\noindent(C2)
Exactly as in the proof of Lemma~{\ref{Lemma_capAcind}}, applied to the cap $B$, we get $\zeta+\beta\in\{0,3,4\}$.
Suppose that $b_1,b_2\notin S$.
Then any $z_{j}\in S$ would have $\vert N(z_{j})\cap S\vert \le 1$, which is impossible, so $\zeta+\beta=0$.
Suppose that exactly one of $b_1,b_2$ lies in $S$.
Then exactly two of $z_{1},z_{2},z_{3}$ lie in $S$, and each has only one neighbour in $B\cap S$, so its predecessor belongs to $S$ and $\zeta+\beta=3$.
Suppose that $b_1,b_2\in S$.
Then again exactly two of $z_{1},z_{2},z_{3}$ lie in $S$, but now each already
has two neighbours in $B\cap S$, so its predecessor does not belong to $S$ and $\zeta+\beta=4$.
In all three cases at most two of $z_{1},z_{2},z_{3}$ belong to $S$.

\medskip

First we prove that $\tau\le 8$.
If $\zeta+\beta=0$, then $\tau=\delta_{1}+\delta_{2}\le 3+3=6$.
So suppose that $\zeta+\beta=3$.
Then $|\{z_1,z_2,z_3\}\cap S|=2$, by (C2).
Moreover, the predecessor of every vertex in $\{z_1,z_2,z_3\}\cap S$ is in $S$, again by (C2).
This means that at least one vertex of $\Delta_2$ is in $S$ and its neighbour outside $\Delta_2$ is in $S$, too.
Thus $\delta_2=2$ and therefore $\tau\le 3+2+3=8$.

Finally, suppose that $\zeta+\beta=4$.
By (C2), the predecessors of the two vertices of $\{z_{1},z_{2},z_{3}\}\cap S$ are not in $S$.
So $\delta_1+\delta_2\le 4$ and $\tau\le 4+4=8$.

\medskip
We distinguish three cases according to $s_k$.

{\bf Case 1}: $s_{k}=R$.
Two strands are in the state $\mathsf{N}_{2}$, and one is in the state $\mathsf{X}$.
Hence, $|\{u^1_k,u^2_k,u^3_k\}\cap S|=2$.
We claim that $\delta_1=0$.
If $u^3_k\notin S$, then $u^1_{k+1},u^2_{k+1}\notin S$, so $\delta_{1}\le 1$.
Hence $\delta_{1}=0$ by (C1).
If $u^3_k\in S$, then one of $u^1_{k+1},u^2_{k+1}$ is not in $S$ since its predecessor has state $\mathsf{N}_{2}$, while the other is not in $S$ since two its neighbours are already outside $S$.
Hence, $\delta_1=0$ by (C1).

Now suppose that $u^3_{k+1}\in S$.
This forces the $u^3_k\notin S$.
Then $u^3_{k+1}$ needs both $v_2,v_3\in S$, i.e.\ $\delta_2=3$.
By (C1) this gives $z_2,z_3\notin S$, so at most one of $z_1,z_2,z_3$ lies in $S$.
Hence $\zeta+\beta=0$ by (C2) and $\tau=0+3+0=3$.

So suppose that $u^3_{k+1}\notin S$.
Then $\delta_{2}\in\{0,2\}$ by (C1).
If $\delta_{2}=2$, then $S\cap\Delta_{2}=\{v_2,v_3\}$, which by (C1) forces $z_2,z_3\in S$.
Consequently, $z_1\notin S$ by (C2).
And since predecessors of $z_2,z_3$ are in $S$, we get $\tau\le 0+2+3$ by (C2). 
If instead $\delta_2=0$, then $\tau\le 0+0+5$, and the equality is excluded by (C2).

Altogether $\tau\le 5$ when $s_k=R$.
The value $5$ is attained by $S\cap T=\{v_2,v_3,z_2,z_3,b_1\}$.

{\bf Case 2}: $s_{k}=Q$.
Let $u^x_k$ be in the state $\mathsf{X}$.
Then $u^x_k\notin S$ and $u^y_k\in S$ for $y\in\{1,2,3\}\setminus\{x\}$.
Also, $u^y_{k+1}\in S$ since $u^y_k$ has state $\mathsf{N}_{1}$.

First suppose that $x\in\{1,2\}$.
Without loss of generality we may assume that $x=1$.
Then $u^2_k$ is in the state $\mathsf{N}_{1}$, so $u^2_{k+1}\in S$ and consequently $\delta_1=2$.
This means that $u^1_{k+1}\notin S$, so $S\cap\Delta_1=\{u^2_{k+1},v_1\}$.
Since $u^3_k$ is in the state $\mathsf{N}_{1}$, we have $u^3_{k+1}\in S$ and consequently also $\delta_2=2$.
Hence, $S\cap\Delta_2=\{u^3_{k+1},v_j\}$ for some $j\in\{2,3\}$.
By (C1) this implies that $z_1,z_j\in S$ with both predecessors $v_1,v_j\in S$ as well.
Hence, $\zeta+\beta=3$, by (C2) and $\tau=2+2+3=7$.

Now suppose that $x=3$.
Then $u^1_k,u^2_k$ are in the state $\mathsf{N}_{1}$, so $u^1_{k+1},u^2_{k+1}\in S$ and consequently $v_1\notin S$, which gives $\delta_1=2$.
Now consider $\zeta+\beta$, see (C2).
If $\zeta+\beta=0$ then $\tau\le 2+3+0<7$.
If $\zeta+\beta=4$, then $|\{z_1,z_2,z_3\}\cap S|=2$, so one of $z_2,z_3$ is in $S$ and its predecessor is not in $S$ since $b_1,b_2\in S$.
Hence, $\delta_2\in\{0,2\}$.
But $\delta_2=2$ implies $u^3_{k+1}\in S$ while $u^3_k\notin S$, a contradiction.
Hence $\delta_2=0$ and $\tau\le 2+0+4<7$.
If $\zeta+\beta=3$ then again $|\{z_1,z_2,z_3\}\cap S|=2$, and $\delta_2\le 2$ since one of $v_2,v_3$ has its successor in $S$.
Thus, $\tau\le 2+2+3=7$.

Altogether $\tau\le 7$ when $s_k=Q$.
The value $7$ is attained by the set $S\cap T=\{u^2_{k+1},u^3_{k+1},v_1,v_2,z_1,z_2,b_1\}$ (when $u^1_k$ is in the state $\mathsf{X}$).

{\bf Case 3:} $s_{k}=P$.
Here $u^1_{k},u^2_{k},u^3_{k}\notin S$, and we have already shown that $\tau\le 8$.
This value is attained.
Take $S\cap T=\{u^1_{k+1},u^2_{k+1},v_1,v_2,v_3,z_2,z_3,b_1\}$, so that $\delta_1=3$, $\delta_2=2$ and $\zeta+\beta=3$.
Each of $u^1_{k+1},u^2_{k+1},v_1$ has two neighbours in $S$ inside $\Delta_{1}$, and its outside neighbour is not in $S$.
Also, $\{v_2,v_3,z_3,b_1,z_2\}$ induce a 5-cycle and the outside neighbours ($u^3_{k+1},z_1,b_2$) are not in $S$.

\medskip

Finally, $\psi(P)+8=10,$ $\psi(Q)+7=10,$ and $\psi(R)+5=9.$ So $\psi(s_{k})+\tau\leq10$ in
all cases.
\end{proof}

\medskip

We are ready to determine $c_{\mathrm{ind}}(H(k))$.

\begin{theorem}
\label{Tm_cind}
Let $n=4k+16$, where $k\ge 1$.
Then
$$
c_{\mathrm{ind}}(H(k))=n/2+2.
$$
\end{theorem}

\begin{proof}
In $k$ the statement gives $c_{\mathrm{ind}}(H(k))=2k+10$.

First we prove the upper bound.
Let $S$ be an arbitrary good set of $H(k)$.
Combining~(\ref{For_telescope}) with Lemmas~{\ref{Lemma_capAcind}} and~{\ref{Lemma_tailcind}}, we obtain
$$
\left\vert S\right\vert =
\alpha+\lambda_0+\sum_{i=1}^{k}\lambda_{i}+\tau \le
2k + \left(\alpha+\lambda_0-\psi(s_{0})\right) + \left(\psi(s_k)+\tau\right) \le 2k+0+10,
$$
as claimed.

Now we prove that lower bound.
We present a good set of required size, which induces a single cycle.
Put
$$
S=\{a_1\}\cup\{u^1_{i},u^3_{i}:0\leq i\leq k+1\}\cup\{v_1,v_3,z_1,z_3,b_1\},
$$
so that $|S|=1+2(k+2)+5=2k+10$.
Obviously, every vertex of $S$ has two neighbours in $S$.
So one has to check only that every vertex of $S$ has also one neighbour outside $S$, which is straightforward.
Therefore $c_{\mathrm{ind}}(H(k))\ge 2k+10$, which together with the upper bound
gives the claimed value.
\end{proof}

The extremal good set constructed above is a single induced cycle, so the length of a longest induced cycle of $H(k)$ also equals $n/2+2.$ Combining the two theorems we obtain the counterexamples announced in the introduction.

\begin{corollary}
\label{Cor_main}
For every $k\ge 1$, the graph $H(k)$ satisfies
$\gamma_{1,2}(H(k)) > c_{\mathrm{ind}}(H(k))$
although $H(k)$ is not a trilobite.
Consequently Conjecture~\ref{Con_ET3} is false.
Moreover $c_{\mathrm{ind}}(H(k))=n/2+2$ satisfies the hypothesis of Conjecture~\ref{Con_ET2}, while $\gamma_{1,2}(H(k))=n/2+3>c_{\mathrm{ind}}(H(k))$ violates its conclusion, so Conjecture~\ref{Con_ET2} is false as well.
Both conjectures fail for infinitely many orders $n\equiv0\pmod4,$ $n\ge 20$.
\end{corollary}

\begin{proof}
This is immediate from Theorems~{\ref{Tm_gamma12}} and~{\ref{Tm_cind}}, together with the observation that $H(k)$ is not isomorphic to any $T_{n}$.
Indeed, $H(k)$ has $2$ triangles and $2$ copies of $K_{2,3}$, while in a trilobite the number of triangles plus the number of copies of $K_{2,3}$ is $2$.
\end{proof}

\bigskip

\begingroup\sloppy
\noindent\textbf{Acknowledgments.}~~The authors acknowledge the partial support
by Slovak research grants APVV 22-0005, APVV 23-0076, VEGA 1/0069/23 and
VEGA 1/0011/25, by ARIS projects J1-3002 and J1-70016, program P1-0383,
bilateral Slovenian-Croatian project BI-HR/25-27-004 and the annual work
program of Rudolfovo, by Project KK.01.1.1.02.0027 co-financed by the European
Regional Development Fund, by Croatian Ministry of Science, Education and
Youth through the bilateral Croatian-Slovenian project 2025--26, and by the
NextGeneration EU foundation via IP-UNIST-17 (GEORAZ).
\par\endgroup

\end{document}